\documentclass[a4paper,8pt]{article}
\usepackage[francais]{babel}
\usepackage[T1]{fontenc}
\usepackage{geometry}
\usepackage[latin1]{inputenc}
\usepackage{amsmath,amssymb,mathrsfs,amsthm}
\usepackage{soul}
\usepackage{epsfig}
\usepackage{hyperref}
\usepackage{graphicx}
\usepackage{xypic}
\usepackage{datetime}
\usepackage{fancyhdr}
\fancypagestyle{plain}{
\lhead{}
\rhead{}
}

\newtheorem{theorem}{Théorème}[section]
\newtheorem{proposition}[theorem]{Proposition}
\newtheorem{lemma}[theorem]{Lemme}

\newtheorem{definition}[theorem]{Définition}
\newtheorem{example}[theorem]{Exemple}
\newtheorem{remarque}[theorem]{Remarque}
\newcommand{\vc}{\|\cdot\|}
\newcommand{\C}{\mathbb{C}}

\newcommand{\Z}{\mathbb{Z}}

\newcommand{\lra}{\longrightarrow}
\newcommand{\al}{\alpha}
\newcommand{\la}{\lambda}

\newcommand{\R}{\mathbb{R}}
\newcommand{\cl}{\mathcal{C}^\infty}
\newcommand{\p}{\mathbb{P}}
\newcommand{\eps}{\varepsilon}
\newcommand{\X}{\mathcal{X}}

\newcommand{\Q}{\mathbb{Q}}
\newcommand{\N}{\mathbb{N}}
\newcommand{\z}{\overline{z}}
\newcommand{\pt}{\partial}

\title{Extension de la torsion analytique holomorphe aux fibrés en droites intégrables}
\date{}
\author{Mounir Hajli}

\begin{document}
\maketitle

\begin{abstract}
Soit $X$ une variété kählerienne compacte. On montre que  la notion de métrique de Quillen  s'étends aux métriques intégrables sur
$X$. En particulier, on établit que  la notion de torsion analytique holomorphe s'étends  à l'ensemble des
fibrés en droites intégrables $\overline{L}$ sur $X$, qui
vérifient $H^q(X,L)=0$
pour tout $q\geq 1$.
\end{abstract}

\tableofcontents

\section{Introduction}
Dans \cite{RaySinger}, Ray et Singer associent à toute  variété kählérienne compacte $(X,\omega)$ et $\overline{E}$ un fibré
hermitien de classe $\cl$ sur $X$,  un réel noté $T\bigl((X,\omega) ;\overline{E}\bigr)$ appelé la torsion analytique holomorphe, défini en posant:
\[
 T\bigl((X,\omega) ;\overline{E}\bigr)=\sum_{q\geq 0}q(-1)^{q+1}\zeta_{\Delta^q_{\overline{E}}}'(0),
\]
où $\zeta_{\Delta^q_{\overline{E}}}'(0)$ est la dérivée en zéro du prolongement analytique de la fonction Zêta $\zeta_{\Delta^q_{\overline{E}}}$ associée au spectre de l'opérateur Laplacien $\Delta^q_{\overline{E}}$ agissant sur $A^{(0,q)}(X,E)$, l'espace des $(0,q)$-formes de classe $\cl$ à coefficients dans $E$, pour tout $q\geq 0$.\\

Dans ce texte, on étend la notion de métrique de Quillen aux métriques admissibles et plus généralement aux métriques intégrables  sur les fibrés en droites holomorphes définis sur une variété kählérienne compacte. Rappelons qu'une métrique admissible sur $L$, un fibré en droites holomorphe, est par définition une limite uniforme d'une suite de
métriques positives de classe $\cl$ sur $X$. Ce sont donc des métriques continues, mais qui sont en général non $\cl$. On ne peut
pas donc appliquer directement la construction de \cite{RaySinger} pour leur associer  une torsion analytique holomorphe, mais on
procède différemment en utilisant une méthode d'approximation moyennant les formules des anomalies, qui  donnent la variation de la
métrique Quillen et par conséquent celle de la torsion analytique,
 en fonction de la variation de la métrique sur $L$.

 Notre résultat
principal s'énonce donc comme suit, voir (théorème \eqref{TAH}):

\begin{theorem}Soit $X$ une variété  complexe kählérienne compacte de dimension $N$ muni d'une forme  de  Kähler $\omega$  et   $\overline{L}=(L,\|\cdot\|)$ un fibré en droites intégrable sur $X$. Pour toute décomposition de $\overline{L}=(E_1,\vc_1)\otimes (E_2,\vc_2)^{{}^{-1}}$ en fibrés admissibles et pour tout choix de$(\vc_{i,n})_{n\in\mathbb{N}}$ une suite de métriques positives $C^\infty$ sur $E_i$ qui converge uniformément vers $\vc_i$, $i=1,2$, la suite double:

\begin{equation}
\Bigl(h_{Q,{\footnotesize{(X,\omega);(E_1\otimes E^{-1}_2,\vc_{1,n}\otimes \vc^{-1}_{2,m})}}}\Bigr)_{n,m\in\mathbb{N}}\footnote{ {\small{On dira
qu'une suite double $(x_{j,k})_{j\in \N,k\in\N}$ converge vers une limite $l$, si
 $\forall\eps>0$, $\exists A\in \R$, $\forall\, n,m>A$ alors $|x_{n,m}-l|<\eps$.}}},
 \end{equation}

est convergente et  la limite ne dépend pas ni de la décomposition ni
de la suite choisie, on l'appellera la métrique de Quillen  généralisée et on la notera par:
 \[h_{Q,{\footnotesize{(X,\omega);(L,\vc)}}}.
\]

Si $H^q\bigl(X,L\bigr)=0$, pour tout $q\geq 1$, alors la suite suivante:

\begin{equation}
\Bigl(T\bigl((X,\omega);(E_1\otimes E^{-1}_2,\vc_{1,n}\otimes \vc^{-1}_{2,m})\bigl)\Bigr)_{n,m\in\mathbb{N}},
 \end{equation}
 converge vers une limite finie.
On l'appellera la \textit{torsion analytique holomorphe de Ray-Singer généralisée } et on la notera par:
 \[
T\bigl((X,\omega),(L,\vc)\bigr).
\]
\end{theorem}

Lorsque $X$ est une surface de Riemann compacte, alors on obtient un résultat plus général, en effet, on peut
considérer des métriques intégrables sur $X$ et sur $L$ et on étend la notion de métrique Quillen à cette situation:
\begin{theorem}
Soit $X$ une surface de Riemann compacte, et $L$ un fibré en droites sur $X$. Soit $h_{\infty,X}$ (resp. vers $h_{\infty,L}$) une
métrique intégrable sur $X$ (resp. $L$). On note par $\omega_{\infty,X}$ la forme kählérienne associée à $h_{\infty,X}$.
\begin{itemize}
\item On considère une décomposition de $(TX,h_{\infty,X})=\overline{G_1}_\infty\otimes
\overline{G_2}_\infty^{-1}$ en fibrés en droites admissibles, et soit $(h_{n,G_1})_{n\in \N} $ (resp.  $(h_{n,G_2})_{n\in \N} )$
une suite de métriques positives et $\cl$ qui converge uniformément vers $h_{\infty,G_1}$ (resp. $h_{\infty,G_1}$). On pose
$h_{n,X}:=h_{n,G_1}\otimes h_{n,G_2}^{-1}$ pour tout $n\in \N$, et  on note par $\omega_{n,X}$ la forme kählérienne associée pour tout
$n\in \N$.
\item  Soit $\overline{L}=(E_1,\vc_1)\otimes (E_2,\vc_2)^{{}^{-1}}$ une décomposition en fibrés admissibles. On considère
$(\vc_{E_i,n})_{n\in\mathbb{N}}$ une suite de métriques positives $C^\infty$ sur $E_i$ qui converge uniformément vers $\vc_i$, pour
$i=1,2$, et on pose $h_{n,L}:=h_{n,E_1}\otimes h_{n,E_2}^{-1}$ pour tout $n\in \N$.
 \end{itemize}
 Alors la suite double suivante:
\[
\Bigl(T\bigl((X,\omega_{n,X});(L,h_{m,L}) \bigr)\Bigr)_{n\in \N,m \in
\N},
\]
converge vers une limite finie qui ne dépend pas du choix des suites ci-dessus. On la note par
$T\bigl((X,\omega_{\infty,X});(L,h_{\infty,L}) \bigr)$.
\end{theorem}

Dans le paragraphe \eqref{CTE}, On montre à l'aide d'un contre exemple la non-validité du théorème \eqref{TAH} si l'on supprime l'hypothèse de positivité des termes de la  suite du théorème. Plus précisément, on va montrer le résultat suivant:

\begin{theorem} Pour toute   forme de Kähler $\omega_{\p^1}$ sur $\p^1$, et pour tout $c>0$, il existe une suite de métriques $\bigl(h_{c,\delta}\bigr)_\delta$ de classe $\cl$ convergeant uniformément vers la métrique canonique de $\mathcal{O}$ sur $\p^1$ telle que:
\[
\limsup_{\delta \mapsto 0}T\bigl((\p^1,\omega_{\p^1});(\mathcal{O},h_{c,\delta})\bigr) - T\bigl((\p^1,\omega_{\p^1}),(\mathcal{O},h_\infty)\bigr) \leq -2c^2.
 \]
\end{theorem}

La section \eqref{gtavt} est dédiée à l'extension de la notion de faisceaux cohérents métrisés aux métriques canoniques. On introduira la définition suivante:

\begin{definition}  Soit $X$ une variété torique lisse. Soit $\overline{\mathcal{F}}=(\mathcal{F},\overline{E}_\bullet\rightarrow \mathcal{F})$ un faisceau cohérent métrisé sur $X$, on dira que la métrique de $\mathcal{F}$ est intégrable (resp. canonique) si chaque terme de $\overline{E}_\bullet$ est une somme directe orthogonale de fibrés en droites munis de métriques intégrables (resp. de leur métriques canoniques), on le note $\overline{\mathcal{F}}^c$ lorsqu'on considère des métriques canoniques partout.
\end{definition}
On établit la proposition suivante:
\begin{proposition}
Tout fibré vectoriel équivariant sur une variété torique lisse admet une métrique canonique.
\end{proposition}
En suivant \cite{Burgos}, on étend la notion de métrique de Quillen associée à cette classe de métriques généralisée. On termina par énoncer un résultat comparant notre approche avec celle de \cite{Burgos} en dimension $1$, c'est l'objet du théorème \eqref{MQG}.\\

\noindent\textbf{Remerciements}: Cet article fait partie de ma thèse. Je tiens à remercier V.Maillot pour m'avoir proposé ce
sujet si riche, pour ses indications et
son encouragement. Je tiens aussi à remercier  J.I. Burgos pour ses conseils et ses remarques sur ce travail, en
particulier pour la remarque \eqref{remBurgos}, G.Freixas, X. Ma et D.Eriksson.

 \section{La métrique de Quillen et la torsion analytique holomorphe, un rappel}
Soit $(X,\omega)$ une variété kählérienne compacte et $\overline{E}$ un fibré hermitien de classe $\cl$ sur $X$. A cette donnée, on associe pour tout $q\geq 0$, un opérateur $\Delta_{\overline{E}}^{(0,q)}$ agissant sur $A^{(0,q)}(X,E)$. On sait que cet opérateur admet un spectre infini positif et que la fonction Zêta associée se prolonge analytiquement au voisinage de $0$, voir par exemple \cite[§ 9.6]{heat}.

On définit  la torsion analytique holomorphe en posant:
\[
 T\Bigl((X,\omega) ;\overline{E}\Bigr)=\sum_{q\geq 0}q(-1)^{q+1}\zeta_{\Delta^q_{\overline{E}}}'(0).
\]
et on  munit $\lambda(L)=\otimes_{q\geq 0}\det\bigl(H^q(X,L) \bigr)^{(-1)^q}$, le déterminant de cohomologie de $L$, de la métrique suivant:
\[
 h_Q:=h_{L^2}\exp\Bigl(T\bigl((X,\omega);\overline{E}\bigr) \Bigr)
\]

 appelée la métrique de Quillen, où $h_{L^2}$ est la métrique $L^2$ induite par $\omega$ et $h_E$. On rappelle que
  cette construction permet de définir l'image directe pour une submersion entre groupes
 de K-théorie arithmétique, voir \cite[propsition 3.1]{Rossler}.

 On dispose de formules appelées formules des anomalies donnant la variation de la métrique de Quillen lorsque la métrique varie sur $E$ ou sur $X$, voit   \cite[théorèmes 0.2, 0.3]{BGS1}. Lorsque la métrique varie sur $E$, alors on a:

\begin{equation}\label{anomalie1}
 \log h_{Q,(X,\omega),(E,\vc)} -\log h_{Q,(X,\omega),(E,\vc')}=
-\Bigl[\int_{X}\widetilde{\mathrm{\mathrm{ch}}}(E,\|\cdot\|,\|\cdot\|')Td(\overline{TX})\Bigr]^{(\dim_\C X)}\footnote{Si $\eta \in A(X)$, alors  on définit $\Bigl[\int_X \eta \Bigr]^{(d)}$ comme étant $\int_X \eta_d $ où $\eta_d$ est la $d$-ème partie graduée du $\eta$ dans $A(X)$, l'algèbre des  $(\ast,\ast)$-formes différentielles de classe $\cl$.}.
\end{equation}
 La variation associée au changement de métrique sur $X$, elle est donnée par:
\begin{equation}\label{anomalie2}
 \log h_{Q,(X,\omega),(E,\vc)} -\log h_{Q,(X,\omega'),(E,\vc)}=
-\Bigl[\int_{X}ch(E,\|\cdot\|)\widetilde{Td}(TX,h_X,h'_X)\Bigr]^{(\dim_\C X)}
\end{equation}
où $\widetilde{\mathrm{\mathrm{ch}}}(E,\|\cdot\|,\|\cdot\|')$ (resp. $\widetilde{Td}(TX,h_X,h'_X)$) est la classe de Bott-Chern associée à la suite
$0\rightarrow (E,\vc)\rightarrow (E,\vc')\rightarrow 0$ (resp. $0\rightarrow (TX,h_X)\rightarrow (TX,h_X')\rightarrow 0$ ) , et au
caractère $ch$ (resp. $Td$), voir
\cite{Character} pour la définition et les propriétés de la classe de Bott-Chern.

\section{Métriques admissibles}
Soit $X$ une variété complexe analytique et $\overline{L}=(L,\vc)$ un fibré en droites hermitien muni d'une métrique continue sur $L$.
\begin{definition}
On appelle premier courant de Chern de $\overline{L}$ et on note $c_1\bigl( \overline{L}\bigr)\in D^{(1,1)}(X)$ le courant défini localement par l'égalité:
\[
c_1\bigl(\overline{L}\bigr)=dd^c\bigl( -\log \|s\|^2\bigr),
\]
où $s$ est une section holomorphe locale et ne s'annulant pas du fibré  $L$.
\end{definition}

\begin{definition}
La métrique $\vc$ est dite positive si $c_1\bigl(L,\vc\bigr)\geq 0$.
\end{definition}
\begin{definition}
 La métrique $\vc$ est dite admissible s'il existe une famille $\bigl(\vc_n \bigr)_{n\in \N}$ de métriques positives de classe $\cl$  convergeant uniformément vers $\vc$ sur $L$. On appelle fibré admissible sur $X$ un fibré en droites holomorphe muni d'une métrique admissible sur $X$.
\end{definition}
On dira que $\overline{L}$ est un fibré en droites intégrable s'il existe $\overline{L}_1$ et $\overline{L}_2$ admissibles tels que
\[
 \overline{L}=\overline{L}_1\otimes \overline{L}_2^{-1}.
\]

\begin{example}
Soit $n\in \N^\ast$. On note par $\mathcal{O}(1)$ le fibré de Serre sur $\p^n$ et on le munit de la métrique définie pour toute section méromorphe de $\mathcal{O}(1)$ par:
\[
 \|s(x)\|_\infty=\frac{|s(x)|}{\max(|x_0|,\ldots,|x_n|)}.
\]
Cette  métrique est admissible.
\end{example}
En fait, c'est un cas particulier d'un résultat plus général combinant   la construction Batyrev et Tschinkel  sur une variété torique projective et  la construction de Zhang. Dans la première construction permet d'associer canoniquement à tout fibré en droites sur une variété torique projective complexe une métrique continue notée $\vc_{BT}$ et déterminée uniquement par la combinatoire de la variété, voir \cite[proposition 3.3.1]{Maillot} et \cite[proposition 3.4.1]{Maillot}. L'approche de Zhang est moins directe, elle utilise un endomorphisme équivariant (correspondant à la multiplication par $p$, un entier supérieur à 2) afin de construire par récurrence une suite de métriques qui converge uniformément vers une limite notée $\vc_{Zh,p}$ et qui, en plus, ne dépend pas  du choix de la métrique de départ, voir \cite{Zhang} ainsi que \cite[théorème 3.3.3]{Maillot}. Mais d'après \cite[théorème 3.3.5]{Maillot} on montre que
\[
 \vc_{BT}=\vc_{Zh,p}.
\]
Que l'on appelle la métrique canonique associée à $L$. Notons que lorsque $L$ n'est pas trivial, alors cette métrique  est non $\cl$.\\

% qu'on en rappelle les principaux résultats:
%Soit $\p(\Delta)$ une variété torique projective  complexe défini par un éventail $\Delta$ et $L$ un fibré en droites sur $\p(\Delta)$
%\begin{proposition}[Construction de Batyrev et Tschinkel] Soit $s$ une section holomorphe de $L$ au dessus d'un ouvert $\Omega\subset \p(\Delta)$. Pour tout point $x\in \Omega$, soit $\si\Delta$

%\end{proposition}

\section{Généralisation de la torsion analytique aux fibrés  intégrables
  sur les variétés  kählériennes compactes }

Dans cette section, on étend la notion de torsion analytique holomorphe aux fibrés intégrables sur une variété kählérienne compacte.
Soit $\overline{L}=(L,\vc)$ un fibré intégrable. En considérant $(\vc_n)$, une suite de métriques convenablement choisie, on va
montrer, en utilisant  la formule des anomalies et la théorie de Bedford-Taylor,
 que  la suite formée par des  métriques de Quillen correspondantes, forme une suite de Cauchy.

 Rappelons le théorème suivant, qui sera utilisé dans la suite:

\begin{theorem}\label{BBBTTT} Soit $U$ un ouvert dans une variété analytique complexe et  soient $u_1,\ldots,u_q$ des fonctions plurisousharmoniques continues sur $U$. Soient $u_1^{(k)},\ldots,u_q^{(k)}$, $q$ suites  de fonctions plurisousharmoniques localement bornées sur $U$ et convergeant uniformément sur tout compact de $U$ vers $u_1,\ldots,u_q$ respectivement et $(T_k)_{k\in \N}$ une suite de courants positifs fermés convergeant faiblement vers $T$sur $U$. Alors:
\begin{equation}\label{bt1}
u_1^{(k)}(dd^c u_2^{(k)})\wedge \cdots (dd^c u_q^{(k)})\wedge T_k\xrightarrow[k\mapsto \infty]{}         u_1(dd^c u_2)\wedge \cdots (dd^c u_q)\wedge T
\end{equation}
\begin{equation}\label{bt2}
  \quad(dd^c u_1^{(k)})(dd^c u_2^{(k)})\wedge \cdots (dd^c u_q^{(k)})\wedge T_k \xrightarrow[k\mapsto \infty]{}   (dd^c u_1)(dd^c
u_2)\wedge \cdots (dd^c u_q)\wedge T
\end{equation}
au sens de la convergence faible des courants.
\end{theorem}

\begin{proof}  cf. par exemple \cite{Demailly}.

\end{proof}

\begin{lemma}\label{expression}
Soit $E_1$ et $E_2$ deux fibrés en droites sur $X$. Soit $\vc_1$ et $\vc_1' $ $($resp. $\vc_2,\vc_2'$ $)$ deux métriques $\cl$ sur $E_1$ $($resp.  sur $E_2$  $)$.

On munit $ L:=E_1\otimes (E_2)^{-1}$  des métriques $\vc_L:= \vc_1\otimes (\vc_2)^{-1}$ et  $\vc'_L:=\vc'_1\otimes (\vc^{'}_2)^{-1}$.  Alors
$\widetilde{\mathrm{\mathrm{ch}}}(L,\vc_L,\vc_L')$ est une somme linéaire de termes de la forme$:$
\begin{equation}\label{terme}
\Bigl(\log\bigl(\vc_1\otimes \vc'_2 \bigr)^2-\log\bigl(\vc'_1\otimes \vc_2\bigr)^2   \Bigr) c_1(\overline{E}_{1})^i c_1(\overline{E}_{1}')^jc_1(\overline{E}_{2})^kc_1(\overline{E}_{2}')^l
,\end{equation}
avec $(i,j,k,l)\in \N^4$.
\end{lemma}

\begin{proof}
 C'est une conséquence directe du \cite[proposition 4.1]{Mounir2}.
\end{proof}

Soit  $\overline{L}$ un fibré en droites intégrable sur $X$. Soit     $(E_1,\vc_1)\otimes (E_2,\vc_2)^{{}^{-1}}$ une décomposition de $\overline{L}$ en fibrés admissibles. Par définition, il existe  $\bigl(\vc_{i,n}\bigr)_{n\in\mathbb{N}}$ une suite de métriques positives $C^\infty$ sur $E_i$ qui converge uniformément vers $\vc_i$ sur $X$, pour $i=1,2$.\\

 On pose, pour tout $(i,j,k,l)\in \N^4$,  $(n,n')\in \N^2 $ et $(m,m')\in \N^2 $: \[\begin{split}
&T_{(n,n'),(m,m')}^{i,j,k,l}(\cdot,\cdot):=\\
\bigl(\log(\vc_{{}_{1,n}}\otimes &\vc_{{}_{2,m'}} )^{{}^2}-\log(\vc_{{}_{1,n'}}\otimes \vc_{{}_{2,m}})^{{}^2}   \bigr) c_1(\overline{E}_{{}_{1,n}})^i c_1(\overline{E}_{{}_{1,n'}})^jc_1(\overline{E}_{{}_{2,m}})^kc_1(\overline{E}_{{}_{2,m'}})^l,
\end{split}
\]
où l'on a choisit implicitement des sections locales holomorphes de $E_1$ et $E_2$ de façon à ce que la forme ci-dessus soit définie sur $X$ entier. \\
%Par  le lemme \eqref{expression}, il suffit de montrer que, $\forall (i,j,k,l)$, $T_{(n,n'),(m,m')}^{i,j,k,l}$
%converge au sens de la convergence faible des courants vers 0. Or  cela résulte de \eqref{Bedford-Taylor}.

%\begin{align*}
% \int_X\widetilde{\mathrm{\mathrm{ch}}}\bigl(E_1\otimes E_2^{-1},\vc_{{}_{1,n}}\otimes \vc_{{}_{2,m}}^{-1},  \vc_{{}_{1,n'}}\otimes \vc_{{}_{2,m'}}^{-1} \bigr)Td(\overline{TX})
%\end{align*}

 %On pose, pour tout $(i,j,k,l)\in \N^4$,  $(n,n')\in \N^2 $ et $(m,m')\in \N^2 $: \[\begin{split}
%&T_{(n,n'),(m,m')}^{i,j,k,l}:=\\
%\bigl(\log(\vc_{{}_{1,n}}\otimes &\vc_{{}_{2,m'}} )^{{}^2}-\log(\vc_{{}_{1,n'}}\otimes \vc_{{}_{2,m}})^{{}^2}   \bigr) c_1(\overline{E}_{{}_{1,n}})^i c_1(\overline{E}_{{}_{1,n'}})^jc_1(\overline{E}_{{}_{2,m}})^kc_1(\overline{E}_{{}_{2,m'}})^l
%\end{split}
%\]
On se propose de montrer que
\[
\Biggl[ \int_X\widetilde{\mathrm{\mathrm{ch}}}\bigl(E_1\otimes E_2^{-1},h_{{}_{1,n}}\otimes h_{{}_{2,m}}^{-1},  h_{{}_{1,n'}}\otimes h_{{}_{2,m'}}^{-1} \bigr)Td(\overline{TX})\Biggr]^{(dim_\C X)},\]tends vers zéro lorsque $n,n',m$ et $m$ tendent vers $\infty$.\\

Par compacité de $X$, \cite[proposition 4.1]{Mounir2} et \eqref{expression}, il existe un ensemble fini $\Omega$,   $\bigl(U_\al \bigr)_{\al\in \Omega}$  un recouvrement ouvert de $X$ et  $\bigl(s_{\al,1}\bigr)_{\al\in \Omega}$ (resp. $\bigl(s_{\al,2}\bigr)_{\al\in \Omega}$) un ensemble de sections locales  holomorphes de $E_1$ (resp. de $E_2$) avec que $s_{\al,1}$ (resp. $s_{\al,2}$) soit non nulle sur $U_\al,\; \forall \al\in \Omega$ tels que pour $\al \in \Omega$, la classe de Bott-Chern $\widetilde{\mathrm{ch}}\bigl(E_1\otimes E_2^{-1},h_{{}_{1,n}}\otimes h_{{}_{2,m}}^{-1},  h_{{}_{1,n'}}\otimes h_{{}_{2,m'}}^{-1} \bigr)$ soit donnée sur $U_\al$ par une combinaison linéaire en:

\[
 T_{(n,n'),(m,m')}^{i,j,k,l}\bigl(s_{\al,1},s_{\al,2}\bigr),
\]
qu'on rappelle égal à:
\begin{align*}
&T_{(n,n'),(m,m')}^{i,j,k,l}(s_{\al,1},s_{\al,2}):=\\
\Bigl(\log(\|s_{1,\al}\|_{{}_{1,n}}\otimes &\|s_{2,\al}\|_{{}_{2,m'}} )^{{}^2}-\log(\|s_{1,\al}\|_{{}_{1,n'}}\otimes \|s_{2,\al}\|_{{}_{2,m}})^{{}^2}   \Bigr) c_1(\overline{E}_{{}_{1,n}})^i c_1(\overline{E}_{{}_{1,n'}})^jc_1(\overline{E}_{{}_{2,m}})^kc_1(\overline{E}_{{}_{2,m'}})^l \; \text{sur}\; U_\al,\, \forall\,\al\in \Omega.
\end{align*}

On considère  $\bigl(\rho_\al\bigr)_{\al\in \Omega}$, une partition de l'unité subordonnée au recouvrement  $\bigl(U_\al \bigr)_{\al\in \Omega}$, c'est à dire que
\begin{enumerate}
\item $\forall \al\in \Omega$, $\rho_\al$ est une fonction réelle de classe $\cl$ sur $X$ à support inclus dans $U_\al$ et à valeurs dans $[0,1]$. 
\item $\sum_{\al\in \Omega}\rho_\al(x)=1$, $\forall x\in X$.
\end{enumerate}
On a donc,
\begin{align*}
 \int_X\widetilde{\mathrm{ch}}\bigl(E_1\otimes E_2^{-1},h_{{}_{1,n}}\otimes h_{{}_{2,m}}^{-1},  h_{{}_{1,n'}}\otimes h_{{}_{2,m'}}^{-1} \bigr)Td(\overline{TX})&=  \sum_{\al\in \Omega}\int_X\rho_\al\,\widetilde{\mathrm{ch}}\bigl(E_1\otimes E_2^{-1},h_{{}_{1,n}}\otimes h_{{}_{2,m}}^{-1},  h_{{}_{1,n'}}\otimes h_{{}_{2,m'}}^{-1} \bigr)Td(\overline{TX})\\
&=\sum_{\al\in \Omega}\int_{U_\al}\rho_\al\,\widetilde{\mathrm{ch}}\bigl(E_1\otimes E_2^{-1},h_{{}_{1,n}}\otimes h_{{}_{2,m}}^{-1},  h_{{}_{1,n'}}\otimes h_{{}_{2,m'}}^{-1} \bigr)Td(\overline{TX}).
\end{align*}
Dans $A(X)=\oplus_{p\in \N}A^{(p,p)}(X)$, on écrit
\[
 Td(\overline{TX})=\sum_{r\geq 0}t_r,
\]
où $t_r\in A^{(r,r)}(X)$, $\forall r\geq 0$.

 Fixons maintenant $\al\in \Omega$, on a
\begin{align*}
 \int_{U_\al}\rho_\al\,\widetilde{\mathrm{ch}}\bigl(E_1\otimes E_2^{-1},h_{{}_{1,n}}\otimes h_{{}_{2,m}}^{-1},  h_{{}_{1,n'}}\otimes
 h_{{}_{2,m'}}^{-1} \bigr)&Td(\overline{TX})=\int_{U_\al}\sum_{r\geq 0}\rho_\al\,\widetilde{\mathrm{ch}}\bigl(E_1\otimes
 E_2^{-1},h_{{}_{1,n}}\otimes h_{{}_{2,m}}^{-1},  h_{{}_{1,n'}}\otimes h_{{}_{2,m'}}^{-1} \bigr)t_r\\
&=\sum_{r\geq 0}\int_{U_\al}\rho_\al\,\widetilde{\mathrm{ch}}\bigl(E_1\otimes E_2^{-1},h_{{}_{1,n}}\otimes h_{{}_{2,m}}^{-1},
h_{{}_{1,n'}}\otimes h_{{}_{2,m'}}^{-1} \bigr)t_r\\
&=\sum_{r\geq 0}\int_{U_\al}\widetilde{\mathrm{ch}}\bigl(E_1\otimes E_2^{-1},h_{{}_{1,n}}\otimes h_{{}_{2,m}}^{-1},  h_{{}_{1,n'}}\otimes
h_{{}_{2,m'}}^{-1} \bigr)\bigl(\rho_\al t_r\bigr).
\end{align*}
Sur $U_\al$, on a les suites de  fonctions suivantes $ \Bigl(-\log(\|s_{1,\al}\|_{{}_{1,n}})\Bigr)_{n\in \N},
\Bigl(-\log\bigl(\|s_{1,\al}\|_{{}_{1,n'}}\bigr)\Bigr)_{n'\in \N}\, \Bigl(-\log(\|s_{2,\al}\|_{{}_{2,m}})\Bigr)_{m\in \N},\quad $ et
$ \Bigl(-\log(\|s_{2,\al}\|_{{}_{2,m'}})\Bigr)_{m'\in \N}$ restreintes à $U_\al$ vérifient les hypothèses du théorème \eqref{BBBTTT}.
 En remarquant que:
\[
 \rho_\al t_r\in A_c(U_\al)\footnote{ $A_c(U_\al)$ désigne l'ensemble des forme différentielles sur $U_\al$ à support compact}\quad \forall\, r.
\]
On déduit que:
\[
\forall i,j,k,l,\quad \biggl[T_{(n,n'),(m,m')}^{i,j,k,l}\bigl(s_{\al,1},s_{\al,2}\bigr)\Bigl(\rho_\al t_r \Bigr)\biggr]^{(\dim_\C X)}\xrightarrow[n,n',m,m'\mapsto \infty]{} 0,
\]
lorsque $n,n',m$ et $m'$ tendent vers $\infty$. En particulier, on obtient à l'aide du lemme \eqref{expression}:
\[
\Biggl[\int_{U_\al}\widetilde{\mathrm{ch}}\bigl(E_1\otimes E_2^{-1},h_{{}_{1,n}}\otimes h_{{}_{2,m}}^{-1},  h_{{}_{1,n'}}\otimes h_{{}_{2,m'}}^{-1} \bigr)\bigl(\rho_\al t_r\bigr)\Biggr]^{(\dim_\C X)}\xrightarrow[n,n',m,m'\mapsto \infty]{}0.
\]
On conclut que
\[
 \biggl[\int_X\widetilde{\mathrm{ch}}\bigl(E_1\otimes E_2^{-1},h_{{}_{1,n}}\otimes h_{{}_{2,m}}^{-1},  h_{{}_{1,n'}}\otimes h_{{}_{2,m'}}^{-1} \bigr)Td(\overline{TX})\biggr]^{(\dim_\C X)},
\]
tends vers $0$, lorsque $n,n',m$ et $m'$ tendent vers $\infty$.\\

Maintenant, on suppose que $(G_1,\vc_{G_1})\otimes (G_2,\vc_{G_2})^{{}^{-1}}$ est une autre décomposition   de $\overline{L}$ en fibrés admissibles. Remarquons que $ E_1\otimes G_2=E_2\otimes G_1$ et qu'il est muni de deux métriques $\vc_{{}_{E_1}}\otimes \vc_{{}_{G_2}}$ et $\vc_{{}_{G_1}}\otimes \vc_{{}_{E_2}}$.

 Pour $i=1,2$, on considère  $\bigl(\vc_{E_i,n}\bigr)_{n\in \N}$ (resp. $\bigl(\vc_{G_i,n}\bigr)_{n\in \N}$) une suite de métriques positives de classe $\cl$ convergeant uniformément vers $\vc_{E_i}$ (resp. vers $\vc_{G_i}$) sur $X$. D'après \eqref{expression}, \[\widetilde{\mathrm{ch}}\bigl(L, \vc_{E_{1,n}}\otimes \vc_{E_{2,n'}}^{-1}, \vc_{G_{1,m}}\otimes \vc_{G_{2,m'}}^{-1}\bigr),\] est une combinaison linéaire des termes de la forme suivante:
\[
\Bigl(\log\bigl(\vc_{E_1,n}\otimes \vc_{G_2,m'} \bigr)^2-\log\bigl(\vc_{G_1,m}\otimes \vc_{E_2,n'}\bigr)^2   \Bigr) c_1(\overline{E}_{n,1})^i c_1(\overline{E}_{2,n'})^jc_1(\overline{G}_{1,m})^kc_1(\overline{G}_{2,m'})^l.\]
Si l'on considère $\eta$ une forme différentielle de degré $(\dim_\C X-p ,\dim_\C X-p)$,
avec $p=i+j+k+l$, sur $U_\al$ à support  compact. Alors par le théorème \eqref{BBBTTT}, on a:
\[
\biggl[\int_{U_\al} \Bigl(\log\bigl(\vc_{E_1,n}\otimes \vc_{G_2,m'} \bigr)^2-\log\bigl(\vc_{G_1,m}\otimes \vc_{E_2,n'}\bigr)^2
\Bigr) c_1(\overline{E}_{n,1})^i c_1(\overline{E}_{2,n'})^jc_1(\overline{G}_{1,m})^kc_1(\overline{G}_{2,m'})^l \eta\biggr]^{(\dim_\C
X)},
\]
converge vers
\[
\biggl[\int_{U_\al} \Bigl(\log\bigl(\vc_{E_1}\otimes \vc_{G_2} \bigr)^2-\log\bigl(\vc_{G_1}\otimes \vc_{E_2}\bigr)^2   \Bigr) c_1(\overline{E}_1)^i c_1(\overline{E}_2)^jc_1(\overline{G}_1)^kc_1(\overline{G}_2)^l \eta\biggr]^{(\dim_\C X)},
 \]
 lorsque $n,n',m$ et $m'$ tendent vers l'infini. Or,
\[
 \overline{L}=(G_1,\vc_{G_1})\otimes (G_2,\vc_{G_2})^{{}^{-1}}=(E_1,\vc_{E_1})\otimes (E_2,\vc_{E_2})^{{}^{-1}},
\]
donc,
\[
\biggl[\int_{U_\al} \Bigl(\log\bigl(\vc_{E_1}\otimes \vc_{G_2} \bigr)^2-\log\bigl(\vc_{G_1}\otimes \vc_{E_2}\bigr)^2   \Bigr) c_1(\overline{E}_1)^i c_1(\overline{E}_2)^jc_1(\overline{G}_1)^kc_1(\overline{G}_2)^l \eta\biggr]^{(\dim_\C X)}=0.
 \]

On va appliquer ce résultat pour étendre la notion de torsion analytique holomorphe aux fibrés intégrables:
\begin{theorem}\label{TAH}Soit $X$ une variété  complexe kählérienne compacte de dimension $N$ muni d'une forme  de  Kähler $\omega$  et   $\overline{L}=(L,\|\cdot\|)$ un fibré en droites intégrable sur $X$.

 Pour toute décomposition de $\overline{L}=(E_1,\vc_1)\otimes (E_2,\vc_2)^{{}^{-1}}$ en fibrés admissibles et pour tout choix de $\bigl(\vc_{i,n}\bigr)_{n\in\mathbb{N}}$ une suite de métriques positives $C^\infty$ sur $E_i$ qui converge uniformément vers $\vc_i$ pour $i=1,2$, la suite double de métriques de Quillen:

\begin{equation}
\Bigl(h_{Q,{\footnotesize{(X,\omega);(E_1\otimes E^{-1}_2,\vc_{1,n}\otimes \vc^{-1}_{2,m})}}}\Bigr)_{n,m\in\mathbb{N}},
 \end{equation}

est convergente et  la limite ne dépend pas ni de la décomposition ni
de la suite choisie, on l'appellera la métrique de Quillen  généralisée et on la notera
 \[h_{Q,{\footnotesize{(X,\omega);(L,\vc)}}}.
\]

Si $H^q\bigl(X,L\bigr)=0$, pour tout $q\geq 1$, alors la suite suivante:

\begin{equation}
\Bigl(T\Bigl(\bigl(X,\omega\bigr);\bigl(E_1\otimes E^{-1}_2,\vc_{1,n}\otimes \vc^{-1}_{2,m}\bigr)\Bigl)\Bigr)_{n,m\in\mathbb{N}},
 \end{equation}
 converge vers une limite finie.
On l'appellera la \textit{torsion analytique holomorphe de Ray-Singer généralisée } et on la notera
 \[
T\bigl((X,\omega),(L,\vc)\bigr),
\]
 et on a pour toute métrique $\cl$,   $\vc'$  sur $L$:
 \[\begin{split}
T\bigl((X,\omega),\overline{L}\bigr)&=T\bigl((X,\omega),\overline{L}'\bigr)+\int_{X}\widetilde{\mathrm{ch}}\bigl(L,\vc,\|\cdot\|'\bigr)Td(\overline{TX})
 -\log\biggl(\frac{h_{L^2,(X,\omega),(L,\vc')}}{h_{L^2,(X,\omega),(L,\vc)}}\biggr),
\end{split}
\]
où   $\widetilde{\mathrm{ch}}\bigl(L,\vc,\|\cdot\|'\bigr)$ ici est une forme différentielle généralisée au sens de \cite[§ 4.3]{Maillot}.

\end{theorem}

\begin{proof} Soit $\bigl(L,\vc\bigr)$ un fibré en droites intégrable sur $X$. Soient $\overline{E}_1$ et $\overline{E}_2$ deux
fibrés en droites admissibles tels que $\overline{L}=\overline{E}_1\otimes \overline{E}_2^{-1}$.
 On pose $\vc_n:=\vc_{1,n}\otimes\vc_{2,n}^{-1}$ pour tout $n\in \N $ où $\bigl(\vc_{1,n}\bigr)_{n\in \N}$ (resp.
 $\bigl(\vc_{2,n}\bigr)_{n\in \N}$) est une suite de métriques positives $\cl$ sur $E_1$ (resp. $E_2$) qui converge uniformément
 vers $\vc_{E_1}$ (resp. $\vc_{E_2}$).

 Si l'on  considère $\vc'$  une métrique $\cl$ quelconque sur $L$, alors d'après \eqref{anomalie1}, on a:

\[
 \log h_{Q, (X,\omega);(L,\vc_n)} -\log h_{Q,(X,\omega);(L,\vc')}=-\Bigl[\int_{X}\widetilde{\mathrm{ch}}\bigl(E_1\otimes E_2^{-1},\vc_n,\|\cdot\|'\bigr)Td(\overline{TX})\Bigr]^{(\dim_\C X)}.
\]
Or, on a montré que le terme à droite converge vers une limite finie qui ne dépend ni du choix de la suite ni de la décomposition. Par conséquent, la suite suivante:
\begin{align*}
\Bigl(-\log h_{\footnotesize{Q,(X,\omega),(L,\vc_n)}}\Bigr)_{n\in \N}=\Bigl(\Bigl[\int_{X}\widetilde{\mathrm{ch}}\bigl(E_1\otimes E_2^{-1},\vc_n,\|\cdot\|'\bigr)Td(\overline{TX})\Bigr]^{(\dim_\C X)}-\log h_{\footnotesize{Q,(X,\omega),(L,\vc')}}\Bigr)_{n\in\N},
\end{align*}
converge vers une limite qu'on note par $-\log\vc_{\footnotesize{Q,(X,\omega),(L,\vc)}}$. Si l'on considère
la  forme différentielle généralisée suivante $\widetilde{\mathrm{ch}}\bigl(L,\vc,\vc'\bigr)Td(\overline{TX})$, voir définition \cite[§
4.3]{Maillot}, alors  on dispose d'une formule d'anomlies généralisée en posant:
\[
\log
h_{\footnotesize{Q,(X,\omega),(L,\vc)}}-\log
h_{\footnotesize{Q,(X,\omega),(L,\vc')}}=-\Bigl[\int_{X}\widetilde{\mathrm{ch}}\bigl(L,\vc,\vc'\bigr)Td(\overline{TX})
\Bigr]^{(\dim_\C X)}.
\]

On suppose maintenant que
\begin{equation}\label{h=0}
 H^q\bigl(X,L\bigr)=0\quad \forall \, q\geq 1.
\end{equation}
donc,
\[
\la(L)=\det\bigl(H^0(X,L) \bigr).
\]
On va montrer que
\[
 \bigl(h_{L^2,(X,\omega),(L,\vc_n)}\bigr)_{n\in \N}\xrightarrow[n\mapsto \infty]{}h_{L^2,(X,\omega),(L,\vc)},
\]
ce qui nous permettra de déduire que la suite suivante converge:
\[\begin{split}
\Bigl(T\bigl((X,\omega),\overline{L}'\bigr)+\int_{X}\widetilde{\mathrm{ch}}\bigl(E_1\otimes E_2^{-1},\vc_n,\|\cdot\|'\bigr)Td(\overline{TX})
-\log\bigl(\frac{h_{L^2,(X,\omega),(L,\vc')}}{h_{L^2,(X,\omega),(L,\vc_n)}}\bigr)\Bigr)_{n\in \N}.
\end{split}
\]

 On définit  alors \textit{la torsion analytique holomorphe généralisée} d'un fibré en droites intégrable $\overline{L}$, vérifiant
 l'hypothèse \eqref{h=0},  sur $X$ munie d'une métrique kählérienne $\omega$, en posant:\[\begin{split}
T\bigl((X,\omega),\overline{L}\bigr)&:=T\bigl((X,\omega),\overline{L}'\bigr)+\lim_{n\mapsto\infty}\int_{X}\widetilde{\mathrm{ch}}\bigl(E_1\otimes E_2^{-1},\vc_n,\|\cdot\|'\bigr)Td(\overline{TX})\\
& -\lim_{n\mapsto\infty}\log\biggl(\frac{h_{L^2,(X,\omega),(L,\vc')}}{h_{L^2,(X,\omega),(L,\vc_n)}}\biggr),
\end{split}
\]
 et on peut  vérifier que
\[
T\bigl((X,\omega),\overline{L}\bigr)=T\bigl((X,\omega),\overline{L}'\bigr)+\int_{X}\widetilde{\mathrm{ch}}\bigl(L,\vc,\|\cdot\|'\bigr)Td(\overline{TX})
 -\log\biggl(\frac{h_{L^2,(X,\omega),(L,\vc')}}{h_{L^2,(X,\omega),(L,\vc_n)}}\biggr).
\]

Soit $\bigl(\vc_n \bigr)_{n\in \N}$ une suite de métriques continues qui converge uniformément vers $\vc$ sur $L$. Rappelons que si  $s$ et $t$ deux deux sections globales de $L$ alors
\[
 \bigl( s,t\bigr)_{L^2,n}=\int_X h_n\bigl( s,t\bigr)\omega,
\]
Par polarisation, on se ramène  à  $s=t$.  Comme  la suite $\bigl(\vc_n\bigr)_{n\in \N}$ converge uniformément vers $\vc$, on peut trouver pour tout $\eps$, un entier $N\in \N$ qui ne dépend pas de $s$ tel que
\[
(1-\eps)\bigl(s,s\bigr)_{L^2,\infty}\leq \bigl(s,s)_{L^2,n}\leq (1+\eps)\bigl(s,s\bigr)_{L^2,n}\quad \forall\, n\geq  N,
\]
où on a noté par $\bigl(\cdot,\cdot\bigr)_{L^2,\infty}$ la norme $L^2$ associée à $\vc$ et à $\omega$.

On déduit que la suite de matrices suivante:
\[
 \Bigl(\Bigl(\bigl(s_j,s_k \bigr)_{L^2,n} \Bigr)_{1\leq k,l\leq r}\Bigr)_{n\in \N},\footnote{$r:=\dim_\C H^0(X,L)$.}
\]
converge vers $\Bigl(\bigl(s_j,s_k \bigr)_{L^2,\infty} \Bigr)_{1\leq k,l\leq \dim_\C H^0(X,L)}$ pour une norme matricielle arbitraire et $\bigl\{s_1,s_2,\ldots,s_r\bigr\}$ est une base de $H^0(X,L)$. Donc,
\[
\Bigl(\det\bigl(s_j,s_k \bigr)_{L^2,n} \Bigr)_{1\leq k,l\leq r}\Bigr)_{n\in \N}\xrightarrow[n\mapsto \infty]{}\det\Bigl(\bigl(s_j,s_k \bigr)_{L^2,\infty}\Bigr)_{1\leq k,l\leq r}.
\]
En particulier,
\[
\Bigl( h_{L^2,(X,\omega),(L,\vc_n)}\Bigr)_{n\in \N}\xrightarrow[n\mapsto \infty]{}h_{L^2,(X,\omega),(L,\vc)}.
\]

 \end{proof}

Lorsque $X$ est une surface de Riemann compacte, alors on obtient un résultat plus général. En effet, on peut
considérer des métriques intégrables sur $X$ et sur $L$ et  on étend comme avant   la notion de métrique Quillen à cette situation:
\begin{theorem}\label{torsSurRiem}
Soit $X$ une surface de Riemann compacte, et $L$ un fibré en droites sur $X$. Soit $h_{\infty,X}$ (resp. vers $h_{\infty,L}$) une
métrique intégrable sur $X$ (resp. $L$). On note par $\omega_{\infty,X}$ la forme kählérienne associée à $h_{\infty,X}$.
\begin{itemize}
\item On considère une décomposition de $(TX,h_{\infty,X})=\overline{G_1}_\infty\otimes
\overline{G_2}_\infty^{-1}$ en fibrés en droites admissibles, et soit $(h_{n,G_1})_{n\in \N} $ (resp.  $(h_{n,G_2})_{n\in \N} )$
une suite de métriques positives et $\cl$ qui converge uniformément vers $h_{\infty,G_1}$ (resp. $h_{\infty,G_1}$). On pose
$h_{n,X}:=h_{n,G_1}\otimes h_{n,G_2}^{-1}$ pour tout $n\in \N$, et  on note par $\omega_{n,X}$ la forme kählérienne associée pour tout
$n\in \N$.
\item  Soit $\overline{L}=(E_1,\vc_1)\otimes (E_2,\vc_2)^{{}^{-1}}$ une décomposition en fibrés admissibles. On considère
$(\vc_{E_i,n})_{n\in\mathbb{N}}$ une suite de métriques positives $C^\infty$ sur $E_i$ qui converge uniformément vers $\vc_i$, pour
$i=1,2$, et on pose $h_{n,L}:=h_{n,E_1}\otimes h_{n,E_2}^{-1}$ pour tout $n\in \N$.
 \end{itemize}
 Alors la suite double suivante:
\[
\Bigl(T\bigl((X,\omega_{n,X});(L,h_{m,L}) \bigr)\Bigr)_{n\in \N,m \in
\N},
\]
converge vers une limite finie qui ne dépend pas du choix des suites ci-dessus. On la note par
$T\bigl((X,\omega_{\infty,X});(L,h_{\infty,L}) \bigr)$.
\end{theorem}

\begin{proof}
On procède comme avant, en utilisant la formule \eqref{anomalie2}.
\end{proof}

\subsection{Un contre exemple}\label{CTE}

On va montrer  à l'aide d'un contre exemple la non validité du  théorème \eqref{TAH}, si l'on supprime la condition de la positivité
, en particulier,  la métrique de Quillen  généralisée considérée comme fonction en la métrique n'est pas
continue sur l'espace des métriques intégrables muni de la topologie de la convergence uniforme. \\

 Pour simplifier, on suppose que $X=\mathbb{P}^1$ et que  $L$ est le fibré trivial. On peut adapter notre exemple au cas
 d'une variété kählérienne compacte quelconque. Soient $c>0$, $0<\varepsilon\ll1$, $0<\delta\ll\varepsilon$ et $0<\gamma\ll\varepsilon-\delta  $. On pose $f$ la fonction définie sur $[1-\eps,1-\eps+\delta]  \cup[1-\gamma,1+\gamma]\cup[1+\eps-\delta,1+\eps]$ par:
\begin{equation*}
f(r) = \left\{
\begin{array}{rl}
\frac{c\sqrt{\delta}}{\delta}r-\frac{c\sqrt{\delta}}{\delta}(1-\eps) & \text{si } r\in [1-\eps,1-\eps+\delta]\\
 c\sqrt{\delta}\quad \quad\quad \quad  \quad \quad  & \text{si } r\in [1-\gamma,1+\gamma]\\
-\frac{c\sqrt{\delta}}{\delta}r+\frac{c\sqrt{\delta}}{\delta}(1+\eps) & \text{si } r\in [1+\eps-\delta,1+\eps],\\
\end{array} \right.
\end{equation*}

et on recolle $f$ par des fonctions $\mathcal{C}^\infty$ de façon à obtenir une fonction qui soit  $\mathcal{C}^\infty$ qui coïncide avec $f$ sur $]1-\eps,1-\eps+\delta[  \cup]1-\gamma,1+\gamma[\cup]1+\eps-\delta,1+\eps[ $, à support compact, nulle en $0$ et qu'elle soit de norme sup inférieur à $ 2c\sqrt{\delta}$. On la note par $f_{c,\delta}$. (On peut supposer que que suite de fonction en $\delta$ est décroissante). On  étend  $f_{c,\delta}$ en   une fonction $\mathcal{C}^\infty$ sur $\mathbb{P}^1$ qu'on notera aussi par $f_{c,\delta}$ et on pose alors $\bigl(\mathcal{O},h_{c,\delta}\bigr)$ le fibré trivial hermitien muni de la métrique $h_{c,\delta}$ donnée par $h_{c,\delta}(1,1)=e^{-f_{\delta,c}}$. Puisque $\sup_{\p^1} |f_{c,\delta}|\leq 2c\sqrt{\delta}$, alors
 $(h_{c,\delta}){{}_{\delta}}$ est une suite croissante qui converge uniformément, lorsque  $\delta\mapsto 0$, vers $h_\infty$   (la métrique canonique de $\mathcal{O}$, c'est à dire $h_\infty(1,1)=1$).\\

Soit $\omega$ une forme kählérienne $\cl$ quelconque sur $\p^1$. On considère la métrique de Quillen associée à $h_{c,\delta}$ et à $\omega$. Par \eqref{anomalie1}, on a:
{\allowdisplaybreaks
\begin{align*}
  -T((\p^1,\omega);(\mathcal{O},h_{c,\delta})) +T((\p^1,\omega);(\mathcal{O},h_\infty))&=\int_{\mathbb{P}^1}
  \widetilde{\mathrm{ch}}(\mathcal{O},h_{c,\delta},h_\infty)Td(\overline{T{\mathbb{P}^1}})+
  \log\frac{h_{L^2,(\p^1,\omega),(\mathcal{O},h_{c,\delta})}}{h_{L^2,(\p^1,\omega),(\mathcal{O},h_\infty)}}\\
&=  \frac{1}{2}\int_{\mathbb{P}^1} f_{{}_{c,\delta}} c_1(\overline{T{\mathbb{P}^1}})+ \int_{\mathbb{P}^1} f_{{}_{c,\delta}} dd^c
f_{{}_{c,\delta}}+ \log\frac{h_{L^2,(\p^1,\omega),(\mathcal{O},h_{c,\delta})}}{h_{L^2,(\p^1,\omega),(\mathcal{O},h_\infty)}}\\
&=\int_{\mathbb{P}^1} f_{{}_{c,\delta}} dd^c f_{{}_{c,\delta}}+\frac{1}{2}\int_{\mathbb{P}^1} f_{{}_{c,\delta}}
c_1(\overline{T{\mathbb{P}^1}})+\log\frac{h_{L^2,(\p^1,\omega),(\mathcal{O},h_{c,\delta})}}{h_{L^2,(\p^1,\omega),(\mathcal{O}
,h_\infty)}}.
  \end{align*}}
 Par construction de $f_{c,\delta}$, on a
{\allowdisplaybreaks
\begin{align*}
\int_{\mathbb{P}^1} f_{{}_{c,\delta}} dd^c f_{{}_{c,\delta}}&=\int_{\mathbb{R}^+}f_{{}_{c,\delta}} \frac{1}{r}\frac{\partial}{\partial r}\Bigl(r\frac{\partial f_{{}_{c,\delta}}}{\partial r}\Bigr) r dr\\
&=\Bigl[rf_{{}_{c,\delta}}\frac{\partial f_{{}_{c,\delta}}}{\partial r}\Bigr]_0^{\infty} -\int_{\mathbb{R}^+} r\Bigl(\frac{\partial f_{{}_{c,\delta}}}{\partial r}\Bigr)^2 dr\\
&=-\int_A r\Bigl(\frac{\partial f_{{}_{c,\delta}}}{\partial r}\Bigr)^2dr- \int_{\mathbb{R}^+\setminus A} r\Bigl(\frac{\partial f_{{}_{c,\delta}}}{\partial r}\Bigr)^2dr\quad \text{où}\, A=[1-\eps,1-\eps+\delta]\cup [1+\eps-\delta,1+\eps] \\
&=-\frac{2}{\delta} f_{{}_{c,\delta}}(1)^2 -\int_{\mathbb{R}^+\setminus A} r\Bigl(\frac{\partial f_{{}_{c,\delta}}}{\partial r}\Bigr)^2dr\\
&= -2c^2-     \int_{\mathbb{R}^+\setminus A} r\Bigl(\frac{\partial f_{{}_{c,\delta}}}{\partial r}\Bigr)^2dr.
\end{align*}}
Donc,
\begin{align*}
-T((\p^1,\omega;(\mathcal{O},h_{c,\delta})) +T((\p^1,\omega;(\mathcal{O},h_\infty))\leq -2c^2 +\frac{1}{2}\int_{\mathbb{P}^1}
f_{{}_{c,\delta}}
c_1(\overline{T{\mathbb{P}^1}})+\log\frac{h_{L^2,(\p^1,\omega),(\mathcal{O},h_{c,\delta})}}{h_{L^2,(\p^1,\omega),(\mathcal{O}
,h_\infty)}}.
\end{align*}
 Comme $\sup_{\p^1}|f_{{}_{c,\delta}}|\leq 2c\sqrt{\delta}$, et que $\p^1$ est projectif, alors il existe une constante
 $M>0$ telle que $|\int_{\p^1}f_{{}_{c,\delta}}\,c_1(\overline{T{\mathbb{P}^1}})|\leq M \sqrt{\delta}\, c$\footnote{Il suffit de
 noter
 qu'il  existe $l\gg 1$ indépendant  de $c$ et de $\delta$, tel que $c_1(\overline{T\p^1})+l \,\omega_{FS}$ soit positif.},
 $\forall\,c>0$ et $\forall \,0<\delta\ll
 1$. Par construction,  $h_{c,\delta}\leq
 h_{\infty}$, alors on obtient:
\begin{equation}\label{nhjncj}
-T((\p^1,\omega);(\mathcal{O},h_{c,\delta})) +T((\p^1,\omega);(\mathcal{O},h_\infty))\leq -2c^2 +M\sqrt{\delta}c,\quad
\forall\,c>0\;\forall\, 0<\delta\ll 1.
\end{equation}

\begin{theorem} Pour toute   forme de kählérienne $\omega$, sur $\p^1$, et pour tout $c>0$, il existe une suite de métriques $\bigl(h_{c,\delta}\bigr)_\delta$ de classe $\cl$ convergeant uniformément vers la métrique canonique de $\mathcal{O}$ sur $\p^1$ telle que:
\[
\limsup_{\delta \mapsto 0}T\bigl((\p^1,\omega);(\mathcal{O},h_{c,\delta})\bigr) - T\bigl((\p^1,\omega_{\p^1}),(\mathcal{O},h_\infty)\bigr) \leq -2c^2.
 \]
\end{theorem}
 \begin{proof}
Par \eqref{nhjncj}, on a
\[
\limsup_{\delta\mapsto 0}T((\p^1,\omega);(\mathcal{O},h_{c,\delta})) - T((\p^1,\omega);(\mathcal{O},h_\infty))\leq -2c^2\quad \forall\, c>0.
\]
On en déduit que la suite $ \Bigl(T((\p^1,\omega);(\mathcal{O},h_{c,\delta}))\Bigr)_{\delta>0}$ ne converge pas vers $T((\p^1,\omega);(\mathcal{O},h_\infty))$ lorsque $\delta$ tends vers $0$.
\end{proof}

\begin{remarque}\leavevmode
\rm{
 \begin{enumerate}

 \item On a $h_{c,\delta}$ est invariante par l'action du tore compact de $\p^1$.
 \item
Malgré que la suite $ \Bigl(T((\p^1,\omega);(\mathcal{O},h_{c,\delta}))\Bigr)_{\delta>0}$ ne converge pas  vers  $T(\mathcal{O},h_\infty)$, on
notera  qu'il existe une constante $M'$ telle que  $ -T((\p^1,\omega);(\mathcal{O},h_{c,\delta}))\leq M'$, $\forall\, 0<\delta\ll 1
$ et $\forall \,c>0$. En effet, de \eqref{nhjncj} on déduit que:
\[
-T((\p^1,\omega);(\mathcal{O},h_{c,\delta})) \leq \frac{M^2}{8}-T((\p^1,\omega);(\mathcal{O},h_\infty))=:M' \quad \forall\, 0<\delta\ll 1,\,\forall\, c>0.
 \]
\end{enumerate}
}
\end{remarque}

On établit  ce fait en toute généralité, voir \cite[théorème 1.3]{Mounir6}. Plus précisément, on montre que La torsion analytique holomorphe vue comme fonction en la métrique est
minorée sur l'espace des
métriques intégrables  et invariantes par l'action du tore compact sur un fibré en droites équivariant  sur $\p^1$.

\subsection{Un calcul explicite de la torsion analytique généralisée dans le cas $X=\mathbb{P}^1$}

D'après le théorème \eqref{torsSurRiem}, ou voir  \cite[théorème 2.5]{Mounir2}, on peut considérer le torsion analytique généralisée
associée à $\overline{\mathcal{O}(m)}_\infty$ le fibré $\mathcal{O}(m)$ muni de sa métrique canonique, et $\p^1$ muni de
$\omega_\infty=\frac{i}{2\pi}\frac{dz\wedge d\z}{\max(1,|z|^4)}$.

En utilisant les formules d'anomalies et connaissant la valeur explicite de $T((\p^1,\omega_{FS});\overline{\mathcal{O}(m)}_{FS})$, on  calcule la torsion analytique généralisée  $T((\mathbb{P}^1,\omega_\infty);\overline{\mathcal{O}(m)}_\infty)$, où

\begin{proposition} On a pour tout entier $m\geq 1$:
%\[
 % T\Bigl(\bigl(\mathbb{P}^1,\omega_{F.S}\bigr);\overline{\mathcal{O}(m)}_\infty\Bigr)     =4\zeta'_\Q(-1)-\frac{1}{2}+(m+2)\log 2
 %+\log\biggl(\frac{(m+2)^{m+1}}{\bigl((m+1)!\bigr)^2}\biggr),
 %\]
 %et
\[
  T((\mathbb{P}^1,\omega_\infty\bigr);\overline{\mathcal{O}(m)}_\infty)          =4\zeta'_\Q(-1)-\frac{1}{6}+ \log\biggl(\frac{(m+2)^{m+1}}{\bigl((m+1)!\bigr)^2}\biggr).
\]
\end{proposition}

\begin{proof}
Voir la preuve de \cite[proposition 2.7]{Mounir2}.
\end{proof}

\begin{remarque}\label{remBurgos}
\rm{Lorsque  $\omega$ est une forme
kählérienne invariante par l'action de $\mathbb{S}^1$ et que la métrique de $\mathcal{O}(m)$ est admissible et invariante aussi
 par l'action du tore compact de $\p^1$, alors on  peut donner une expression pour $
 T((\mathbb{P}^1,\omega\bigr);\overline{\mathcal{O}(m)})$  en termes la transformée de Legendre-Fenchel. En fait, on peut
 exprimer les formules d'anomalies, dans ce cas, comme intégrales des transformées de Legendre-Fenchel associées aux
 métriques de $\p^1$ et de $\mathcal{O}(m)$, voir \cite{Burgos2} pour la définition de la transformée de Legendre-Fenchel associée
 à une métrique admissible invariante par l'action du tore compact. }
\end{remarque}
\section{Généralisation de la torsion analytique sur les variétés toriques lisses dans le formalisme de Burgos, Litcanu et
Freixas}\label{gtavt}
\subsection{Extension de la notion de métrique canonique aux faisceaux cohérents metrisés sur une variété torique lisse}

Soit $X$ une variété complexe. On étend la classe de fibrés vetoriels hermitiens de classe $\cl$ en considérant les fibrés hermitiens $\overline{E}$ qui s'écrivent sous la forme:
\[
 \overline{E}=\overline{E'}\oplus \oplus_{i=1}^d \overline{L}_i,
\]
où $\overline{E'}$ est un fibré hermitien de classe $\cl$ et $\overline{L}_1,\ldots,\overline{L}_d$ sont des fibrés en droites intégrables, on appellera $\overline{E}$ un fibré hermitien intégrable.\\

Dans cette partie on utilise la théorie du \cite[§ 4.3]{Maillot} pour étendre la notion de classes de Bott-Chern aux suite exactes de fibrés  hermitiens intégrables. Plus précisément, si $ch$ est le caractère de Chern et
\[
\overline{\eta}: 0\longrightarrow \overline{S}\longrightarrow \overline{E}\longrightarrow \overline{Q}\longrightarrow 0,
\]
est une suite exacte de fibrés hermitiens intégrables, alors on leur associe une  unique forme différentielle généralisée $\widetilde{\mathrm{ch}}(\overline{\eta})$ élément de $\widetilde{\overline{A}_g}^\ast(X)$, (voir \cite[p. 66]{Maillot}), qui vérifie les mêmes propriétés classiques de la classe de Bott-Chern, voir \cite{Character}.

Supposons que $\widetilde{\mathrm{ch}}(\overline{\eta})$ existe. On écrit $\overline{S}=\overline{S'}\oplus \overline{S''}$, $\overline{E}=\overline{E'}\oplus \overline{E''}$ et $\overline{Q}=\overline{Q'}\oplus \overline{Q''}$ où $\overline{S}$, $\overline{E'}$ et $\overline{Q'}$ sont des fibrés hermitiens de classe $\cl$ et que chacun des fibrés hermitiens $\overline{S''}$, $\overline{E''}$ et $\overline{Q''}$ est une somme orthogonale de fibrés en droites intégrables.

\[
\xymatrix{
% &\ar@{.>}[d]& \ar@{.>}[d]&\ar@{.>}[d]&\\
\overline{\eta}:\;0\ar[r]&\overline{S} \ar[r] & \overline{E}  \ar[r]  & \overline{Q} \ar[r] &0\\
\overline{\eta_0}:\;0\ar[r]&\overline{S_0} \ar[r] \ar[u]^{id}  &\overline{E_0} \ar[r]\ar[u]^{id} & \overline{Q_0}\ar[u]^{id}\ar[r] &0
  }
\]
avec $\overline{S_0}=\overline{S'}\oplus \overline{S''_0}$, $\overline{E_0}=\overline{E'}\oplus \overline{E''_0}$ et $\overline{Q_0}=\overline{Q'}\oplus \overline{Q''_0}$ tels que $\overline{S''_0}$, $\overline{E''_0} $ et $\overline{Q''_0}$ sont munis de métriques hermitiennes de classes $\cl$. On a
\[
\widetilde{\mathrm{ch}}(\overline{\eta_0})-\widetilde{\mathrm{ch}}(\overline{\eta})=\widetilde{\mathrm{ch}}(\overline{C_1}\oplus \overline{C_3})-\widetilde{\mathrm{ch}}(\overline{C_2}),
\]
où $\overline{C_1}$, $\overline{C_2}$ et $\overline{C_3}$ désignent respectivement la première, la deuxième et la  troisième colonne du diagramme ci-dessus. On vérifie que
\[
 \widetilde{\mathrm{ch}}(\overline{C_1}\oplus \overline{C_3})=\widetilde{\mathrm{ch}}(\overline{C_1})+ \widetilde{\mathrm{ch}}(\overline{C_3})=\widetilde{\mathrm{ch}}(\overline{S''_0}\rightarrow  \overline{S''})+\widetilde{\mathrm{ch}}(\overline{Q''_0}\rightarrow  \overline{Q''}),
\]
 et
\[
 \widetilde{\mathrm{ch}}(\overline{C_2})=\widetilde{\mathrm{ch}}(\overline{E''_0}\rightarrow \overline{E''}).
\]
Le calcul de ces dernières classes se ramène au cas des fibrés vectoriels de rang 1. On déduit l'unicité et l'existence de la classe $\widetilde{\mathrm{ch}}(\overline{\eta})$. Pour les propriétés de fonctorialité, elles découlent de la théorie des formes différentielles généralisées de \cite[§ 4.3]{Maillot}.

Comme application, on peut étendre la notion de faisceaux cohérents métrisés de \cite{Burgos} pour qu'elle prend en compte les fibrés hermitiens intégrables. On dira que $\overline{\mathcal{F}}=\bigl(\mathcal{F},\overline{E}_\bullet\rightarrow \mathcal{F} \bigr)$ est un faisceau cohérent intégrable si chaque terme du complexe $\overline{E}_\bullet$ est un fibré hermitien intégrable, et on étend la définition de \cite{Burgos} pour la classe de Bott-Chern à notre situation.\\

Dans la suite, on s'intéresse aux variétés toriques. Le cas des variété toriques lisses est plus intéressant, puisqu'on montre dans
\cite{Klyachko}, que tout fibré vectoriel équivariant admet une résolution canonique par des fibrés vectoriels scindés.
 Rappelons d'abord la construction de la métrique canonique sur un fibré en droites équivariant. Soit $\X$ une variété torique
projective lisse sur $\mathrm{Spec}(\Z)$, on note par $T$ le tore associé. Soit $L$ un fibré en droites sur $\X$, on rappelle qu'on
 construit de manière unique une métrique sur $L(\C)$:
\begin{proposition}\label{metriquecanonique}
 Soit $L$ un fibré en droites sur $\X$. Il existe un diviseur horizontal $T$-invariant $D$ sur $\X$ et un isomorphisme:
\[
 \Phi:L\longrightarrow \mathcal{O}(D).
\]
 La métrique $\Phi^\ast \vc_{D,\infty}$ sur $L$ est indépendante des choix de $D$ et $\Phi$. On l'appelle métrique canonique sur $L$ et on la note $\vc_{L,\infty}$. On note $\overline{L}_\infty=\bigl(L,\vc_{D,\infty} \bigr)$ le fibré $L$ muni de sa métrique canonique.
\end{proposition}
\begin{proof}
 Voir \cite[proposition 3.4.1]{Maillot}.
\end{proof}

\begin{definition}\label{canoniquefaisceau}  Soit $X$ une variété torique lisse. Soit
$\overline{\mathcal{F}}=(\mathcal{F},\overline{E}_\bullet\rightarrow \mathcal{F})$ un faisceau cohérent métrisé sur $X$, on dira que
la métrique de $\mathcal{F}$ est intégrable (resp. canonique) si chaque terme de $\overline{E}_\bullet$ est une somme directe
orthogonale de fibrés en droites munis de métriques intégrables (resp. de leur métriques canoniques). On appelle
$\overline{\mathcal{F}}$
fibré intégrable. On note $\overline{\mathcal{F}}^c$ au lieu de $\overline{\mathcal{F}}$ lorsqu'on considère des métriques canoniques partout.
\end{definition}
\begin{proposition}Soit $f:Y\rightarrow X$ un morphisme équivariant de variétés toriques projectives lisses, si $\overline{\mathcal{F}}^c$ est un faisceau cohérent métrisé dont la métrique est canonique, alors $f^\ast \overline{\mathcal{F}}^c$ l'est aussi
 \end{proposition}
\begin{proof}   Cela résulte du cas classique pour les fibrés en droites.
 \end{proof}

\begin{proposition}
Tout fibré vectoriel équivariant sur une variété torique lisse admet une métrique canonique.
\end{proposition}
\begin{proof}
On sait que tout fibré vectoriel équivariant, et plus généralement un faisceau équivariant, sur $X$, voir  \cite{Klyachko} pour le cas de fibré vectoriel équivariant, admet une résolution finie canonique en fibrés vectoriels scindés. Ce qui montre que comme pour le cas des fibrés en droites, on peut associer à tout faisceau équivariant  une métrique canonique généralisée.
\end{proof}

On se propose maintenant de comparer la notion de métriques canonique pour un fibré en droites et la définition \eqref{canoniquefaisceau}.

\begin{proposition}
on va montrer que si $L$ est un fibré en droites sur $X$ alors

\[
\widetilde{\mathrm{ch}}\bigl( \overline{L}^c\longrightarrow \overline{L}_\infty\bigr)=0
\]
 Plus généralement, si $\mathcal{F}=\oplus_{k=1}^e L_k$ est un fibré vectoriel scindé. On munit $\mathcal{F}$ d'une métrique
 canonique c'est à dire $\mathcal{F}^c=(\mathcal{F}, \overline{E}_\bullet\rightarrow \mathcal{F})$ comme dans la définition ci-dessus.
  alors
\[
 \widetilde{c}_1\bigl(\overline{\mathcal{F}}^c\longrightarrow \oplus_{k=1}^e \overline{L}_k \bigr)=0\quad\text{et}\quad
 \widetilde{\mathrm{ch}}_{\max}\bigl(\overline{\mathcal{F}}^c\longrightarrow \oplus_{k=1}^e \overline{L}_{k,\infty} \bigr)=0.
\]
dans $\widetilde{A}(X)$\footnote{$\widetilde{A}(X)$ est par définition l'espace des $(\ast,\ast)$-formes différentielles $\cl$ sur $X$ modulo $\mathrm{Im}\pt+\mathrm{Im}\overline{\pt}$. }.
\end{proposition}
\begin{proof}
On suppose que $E_\bullet=E_l\rightarrow E_{l-1}\rightarrow \cdots \rightarrow E_1$ tel que:
\[
 \eta:0\longrightarrow E_l\longrightarrow E_{l-1}\longrightarrow \cdots \longrightarrow E_1\lra \mathcal{F}\longrightarrow 0,
\]
soit une suite exacte de fibrés vectoriels sur $X$.
On suppose que pour tout $i=1,\ldots,l$, $E_i=\bigoplus_{k=1}^{e_i}L_{i,k}$ où $L_{i,1},L_{i,2},\ldots,L_{i, e_i}$ sont des fibrés
en
droites sur  $X$. On munit alors $E_i$ de la métrique suivante $h_{E_i}$:
\[
 h_{E_i}=h_{L_{i,1,\infty}}\oplus \cdots \oplus h_{L_{i,e_i,\infty}},
\]
où $h_{L_{i,k,\infty}}$ est la métrique canonique de $L_{i,k}$ pour $k=1,\ldots,e_i$ et pour tout $i=0,\ldots,l$. On note par
$\overline{\eta}$ la suite $\eta$ munie de ces métriques.

Si l'on pose $L=\det\bigl(E_l\bigr)\otimes \det\bigl(E_{l-2}\bigr)\otimes \det\bigl(E_{l-[\frac{l}{2}]} \bigr)$ et
$L'=\det\bigl(E_{l-1}\bigr)\otimes \det\bigl(E_{l-3}\bigr)\cdots \det\bigl(E_{l-[\frac{l+1}{2}]} \bigr)$ alors on dispose  d'un
isomorphisme canonique induit par $\eta$:
\[
 \al:L\longrightarrow L'.
\]
Par construction,  $L$ (resp. $L'$) est muni  de la métrique $h_L:=\det\bigl(h_{E_l}\bigr)\otimes \det\bigl(h_{E_{l-2}}\bigr)\otimes
\cdots\det\bigl(h_{E_{l-[\frac{l}{2}]}} \bigr)$ $\Bigl($resp. de  $h_{L'}:=\det\bigl(h_{E_{l-1}}\bigr)\otimes
\det\bigl(h_{E_{l-3}}\bigr)\cdots \det\bigl( E_{l-[\frac{l+1}{2}]}\bigr)\Bigr)$. D'après \eqref{metriquecanonique} et
\cite[proposition 3.3.6]{Maillot} ces deux métriques sont les métriques canoniques de $L$ et $L'$.

On a
\[
h_L=h_{L'}\exp\bigl(\widetilde{c}_1(\overline{\eta})\bigr).
\]
Par unicité de la métrique canonique, voir \eqref{metriquecanonique}, on conclut que
\[
 \widetilde{c}_1(\overline{\eta})=0.
\]

Soit $\X$ un modèle  de $X$ sur $\mathrm{Spec}(\Z)$. On considère, voir \cite[p. 53]{Maillot}:
\begin{align*}
 a:\widetilde{\overline{A}}^{\ast-1,\ast-1}&\longrightarrow \widehat{\mathrm{CH}}_{int}^\ast\bigl(X\bigr)\\
\beta&\longmapsto \bigl[(0,\beta) \bigr],
\end{align*}
 D'après \cite{Maillot}, on a:
\[
a\Bigl( \widetilde{\mathrm{\mathrm{ch}}} \bigl(\overline{\mathcal{F}}^c\longrightarrow \oplus_{k=1}^e \overline{L}_{k,\infty} \bigr)\Bigr)=\widehat{\mathrm{\mathrm{ch}}}\bigl(\overline{\mathcal{F}}^c)-\widehat{\mathrm{\mathrm{ch}}}\bigl(\oplus_{k=1}^e \overline{L}_{k,\infty} \bigr),
\]
où $\widehat{\mathrm{\mathrm{ch}}}$ est la classe de Chern arithmétique associée au caractère $ch$, et $\widehat{\mathrm{\mathrm{ch}}}\bigl(\overline{\mathcal{F}}^c)$
est par
définition égale à $\sum_{j=1}^l (-1)^j \widehat{\mathrm{\mathrm{ch}}}\bigl(\overline{E}_j)$.  Par  \cite[lemme 7.4.2]{Maillot},
les puissances maximales des premières classe arithmétiques de Chern associées aux fibrés en droites munis de métriques canoniques
sont nulles.

Par suite,
\[
\widetilde{\mathrm{\mathrm{ch}}}_{\max} \bigl(\overline{\mathcal{F}}^c\longrightarrow \oplus_{k=1}^e \overline{L}_{k,\infty}
\bigr)=0.
\]

\end{proof}
\subsection{Torsion analytique généralisée associée aux fibrés intégrables sur variété torique}

Soit $X$ une  variété torique non-singulière munie de $\omega$, une forme kählérienne et $\overline{\mathcal{L}}$ un fibré hermitien sur $X$. On
 notera par $h_{Q,\overline{TX},\overline{\mathcal{L}}} $ la métrique de Quillen  au lieu de
 $h_{Q,\omega,\overline{\mathcal{L}}}$, où $\overline{TX}$ est le fibré tangent  muni de la métrique   associée à $\omega$.\\

On munit $\overline{L}_\infty$  de sa métrique canonique.  On pose $\overline{TX}^c=\bigl(TX,\, 0\longrightarrow \overline{\mathcal{O}}_\infty^r\longrightarrow \oplus_{i=1}^{N+r}\overline{\mathcal{O}(D_i)}_\infty\longrightarrow TX \bigr)$  considéré comme un fibré canonique où

 \[\varepsilon_X:0\lra \mathcal{O}_X^r\lra \oplus_{i=1}^{N+r}\mathcal{O}(D_i)\lra TX\lra 0, \]
est la suite d'Euler sur $X$, $r$ est le rang de $\mathrm{Pic}(X)$.

Dans ce cas, on  définit $h_{Q,\overline{TX}^c,\overline{L}_\infty}$, la norme de Quillen associé au fibré $\overline{L}_\infty$
et à $\overline{TX}^c$  en posant:
\[\log h_{Q,\overline{TX}^c,\overline{L}_\infty}:= -\int_Xch\bigl(\overline{L}_\infty\bigr)\widetilde{Td}\Bigl(0 \longrightarrow \overline{TX}^c\longrightarrow \overline{TX}\longrightarrow 0\Bigr)+ \log h_{Q,\overline{TX},\overline{L} }.
\]
 On montre, voir \cite[théorème 2.19]{Burgos}, que:
 \[
\widetilde{Td}\Bigl(0 \longrightarrow \overline{TX}^c\longrightarrow \overline{TX}\longrightarrow 0\Bigr)=\widetilde{Td}(\overline{\varepsilon}_X),
              \]
où \[
    \overline{\varepsilon}_X:0\longrightarrow \overline{\mathcal{O}}_\infty^{\oplus{r}}\longrightarrow \oplus_{i=1}^{N+r}\overline{\mathcal{O}(D_i)}_\infty\longrightarrow \overline{TX}\longrightarrow 0.
   \]
Donc
\[\log h_{Q,\overline{TX}^c,\overline{L}_\infty }=-
 \int_X \mathrm{ch}(\overline{L}_\infty)\widetilde{Td}(\overline{\varepsilon}_X)+ \log h_{Q\overline{TX},\overline{L}_\infty}.\]
On a,
\[
a\bigl(\widetilde{Td}(\overline{\varepsilon}_X)\bigr)=\prod_{i=1}^{N+r}\widehat{Td}\bigl(\overline{\mathcal{O}(D_i)}_\infty\bigr)
-\widehat{Td}(\overline{TX}).
\]

Par le théorème de Riemann-Roch arithmétique, voir \cite{ARR} on a:
\[
\log h_{Q\overline{TX},\overline{L}_\infty}=\int_X \widehat{\mathrm{\mathrm{ch}}}(\overline{ L}_\infty)\,Td^A(\overline{TX}).
\]

On obtient  alors, en utilisant l'annulation des puissance maximales des classes arithmétiques associées aux métriques canoniques:
\begin{align*}
-\log h_{Q,\overline{TX}^c,\overline{L}_\infty }
&=\int_X\widehat{\mathrm{\mathrm{ch}}}(\overline{L}_\infty)\,\biggl(\prod_{i=1}^{N+r}\widehat{Td}
\bigl(\overline{\mathcal{O}(D_i)}_\infty\bigr)-\widehat{Td}(\overline{TX})\biggr)+\int_X\widehat{\mathrm{\mathrm{ch}}}(\overline{
L}_\infty)\,Td^A(\overline{TX})\\
&=-\int_X\widehat{\mathrm{\mathrm{ch}}}(\overline{L}_\infty)\widehat{Td}(\overline{TX})+\int_X\widehat{\mathrm{\mathrm{ch}}}
(\overline{ L}_\infty)\,\widehat{Td}(\overline{TX})-\int_X \mathrm{ch}(L)Td(TX)R(TX).
\end{align*}
Donc,
\[
 \log h_{Q,\overline{TX}^c,\overline{L}_\infty }=\int_X \mathrm{ch}(L)Td(TX)R(TX).
\]

\begin{theorem}\label{MQG}
 Avec les notations précédentes, on a
\[
 h_{Q,\overline{T\p^1}^c,\overline{\mathcal{O}(m)}_\infty }=h_{Q,(\p^1,\omega_\infty);\overline{\mathcal{O}(m)}_\infty}\quad \forall
 m\in \N.
\]

\end{theorem}

\begin{proof}
 Cela résulte de:
\[
 \widetilde{Td}\Bigl(0\lra \overline{T\p^1}^c\lra \overline{T\p^1}_\infty\lra 0 \Bigr)=0,
\]
dans $\widetilde{A}(\p^1)$.
\end{proof}

\bibliographystyle{plain}
\bibliography{biblio}

\vspace{1cm}

\begin{center}
{\sffamily \noindent National Center for Theoretical Sciences, (Taipei Office)\\
 National Taiwan University, Taipei 106, Taiwan}\\

 {e-mail}: {hajli@math.jussieu.fr}

\end{center}

\end{document}